\documentclass[11pt]{amsart}
\usepackage{amscd, amssymb, amsmath, amsthm, amsfonts}
\usepackage{mathrsfs}
\usepackage{amscd, amssymb, amsmath, amsthm, amsfonts}
\usepackage{latexsym, graphics, graphicx, psfrag}
\usepackage{color,xcolor,colortbl}
\usepackage[colorlinks=true]{hyperref}
\usepackage[all]{xy}
\usepackage[small,nohug,heads=vee]{diagrams}
\usepackage{enumerate}
\diagramstyle[labelstyle=\scriptstyle]

\numberwithin{equation}{section}

\newtheorem{theorem}{Theorem}[section]
\newtheorem{lemma}[theorem]{Lemma}
\newtheorem{proposition}[theorem]{Proposition}

\theoremstyle{definition}

\newtheorem{remark}[theorem]{Remark}

\newtheorem{conjecture}[theorem]{Conjecture}

\begin{document}

%\begin{document}

\title[Euler Characteristics of 4-manifolds with 3-manifold groups]
{Minimal Euler Characteristics of 4-manifolds with 3-manifold groups}

\author{Hongbin Sun}
\address{Department of Mathematics, Rutgers University - New Brunswick, Hill Center, Busch Campus, Piscataway, NJ 08854, USA}
\email{hongbin.sun@rutgers.edu}

\author{Zhongzi Wang}
\address{Department of Mathematics Science, Tsinghua University, Beijing, 100080, CHINA}
\curraddr{School of Mathematical Sciences, Peking University, Beijing, 100871, CHINA}
\email{wangzz22@stu.pku.edu.cn}


\thanks{The first author is partially supported by Simons Collaboration Grant 615229.}

\begin{abstract} Let $\pi=\pi_1(M)$ for a compact 3-manifold $M$, and let $\chi_4$, $p$ and $q^*$ be the invariants of Hausmann-Weinberger,
Kotschick and Hillman respectively. 
For certain class of compact 3-manifolds $M$, including all those not containing two-sided $RP^2$,
we determine $\chi_4(\pi)$. We address when  does $p(\pi)=\chi_4(\pi)$, when does $q^*(\pi)=\chi_4(\pi)$, and  answer a question raised by Hillman. \end{abstract}

\date{} 
\maketitle 

\tableofcontents
\section{Introduction} 
For a CW-complex $X$, we use $H_i(X)$, $H^i(X)$, and $\beta_i(X)$ (resp. $H_i(X; \mathbb{Z}_2)$, $H^i(X;\mathbb{Z}_2)$, and $\beta_i(X; \mathbb{Z}_2)$)
to denote its $i$-th homology group, $i$-th cohomology group and  $i$-th Betti number with real coefficient (resp. $\mathbb{Z}_2$ coefficient). 
For a finitely presented group $G$, denote $\beta_i(G)$ ($\beta_i(G; \mathbb{Z}_2)$) to be the $i$-th betti number of $K(G, 1)$, 
the classifying space of $G$. 
Denote $\chi(X)$ to be the Euler characteristic of a finite CW complex $X$, and denote $\sigma(X)$ to be the signature of a closed  oriented 4-manifold $X$. In this paper, by $4$-manifolds, we mean topological $4$-manifolds, although all $4$-manifolds we will construct are smooth manifolds.

In 1985, Hausmann and  Weinberger \cite{HW} introduced the 4-manifold Euler characteristic  for a finitely presented group $G$, defined by  $$\chi_4(G)=\text{inf}\{\chi(X)\ |\ \text{$X$ is a closed orientable 4-manifold and $\pi_1(X)\cong G$}\}.$$ 

There are also some variations of  $\chi_4(G)$. 
In 1994, Kotschick  \cite{Ko1} introduced 
$$p(G)=\text{inf}\{\chi(X)-|\sigma(X)|\ |\ \text{$X$ is a closed orientable 4-manifold and $\pi_1(X)\cong G$}\}.$$
In  2002, Hillman \cite{Hi} introduced
$$q^*(G)=\text{inf}\{\chi(X)\ |\ \text{$X$ is a closed 4-manifold and $\pi_1(X)\cong G$}\}.$$ 
It is known that 
\begin{align}\label{1.1}
\chi_4(G) \ge p(G),  \,\,\, \chi_4(G) \ge q^*(G).
\end{align}

Note both $\chi_4(G)$ and $q^*(G)$ are denoted by $q(G)$ in their original definitions,  see \cite{HW} and \cite{Hi}.
Moreover, $q^*(G)$ is originally defined for $PD_4$-complexes in \cite{Hi}.
%, and Theorem \ref{main1.4} is also valid when $q^*$ is defined for $PD_4$-complexes, see Remark \ref{last} (2). 

Kotschick made
 a useful observation to estimate the lower bound of $\chi_4(G)$ and $p(G)$ when $\beta_4(G)=0$ 
 \cite[Theorems 2.8 and 4.2]{Ko1}. This observation is crucial for our works in this article (c.f. the approach in Section \ref{secondhalf2}).

More variations  and generalizations  of $\chi_4(G)$ can be found in \cite{Hi}, \cite{Ko2}, \cite{BK}, \cite{KL} and \cite{AH}. 

An important family of finitely presented groups which can be classified  are groups of compact 3-manifolds, which include all cyclic groups,  free groups, surface groups and knot groups. Many studies have been made for 4-manifolds with 3-manifold groups, see \cite{JK}, \cite{Ko1}, \cite{Hi}, \cite{KL}, \cite{KLPT} and the references therein.

We survey known results on above invariants of 3-manifold groups in the following theorem.

{\bf Theorem 1.0} {\it Suppose $M$ is a compact 3-manifold and $\pi_1(M)=\pi$.

(1) $\chi_4(\pi)=p(\pi)= 2$ if $M$ is closed, orientable and aspherical,  \cite[Proposition 5.6]{Ko1}.

(2) $\chi_4(\pi)=2-2q$ if $M$ is closed and orientable, where $q$ is the maximal rank of free group in the free product decomposition of $\pi$,  \cite[page 61-62]{Hi} and \cite[Theorem 3.3]{KL}.

(3) $\chi_4(\pi)=0$ if $\pi$ is the group of a knot complement in the 3-sphere $S^3$, \cite[Corollary 3.12.3]{Hi} and \cite[Theorem 3.4]{KL}. 

(4) Suppose $M$ is a closed aspherical 3-manifold and $\pi_1(X)=\pi$ for a closed 4-manifold
$X$. If $X$ and $M$ have the same orientability, then $\chi(X)>0$ and 
$q^*(\pi)\in \{1,2\}$, \cite[Theorem 3.13 and Corollary 3.13.1]{Hi}. }

In \cite[page 63]{Hi}, Hillman asked 
whether results in (4) above can be extended to all closed 3-manifold groups without torsion and without free $\mathbb{Z}$-factors.

In this paper, we try to determine $\chi_4(\pi)$ for fundamental groups of compact  3-manifolds. We also get some results related to
$p(\pi)$ and $q^*(\pi)$.  It turns out that, for non-orientable 3-manifolds,
the problem becomes  more difficult and some new approaches are needed.

For undefined terminologies, see \cite{He} and \cite{Ha2} about 3-manifolds, see \cite{Kir} about 4-manifolds, and see
 \cite{Ha1} about  algebraic topology.

The Kneser-Milnor theorem claims that each compact 3-manifold has a prime decomposition, whose prime factors
are unique up to homeomorphism (up to possibly replacing $S^2\times S^1$ by $S^2\tilde{\times}S^1$) and permutation (\cite{He}).
So $M$ has a  prime
decomposition 
\begin{align}\label{1.2}
M=(\#_{i=1}^mM_i)\#(\#_{j=1}^n N_j)\#(\#_{l=1}^p Q_l)\#(\#_{e=1}^q S_e).
\end{align}
Here each prime factor may or may not be orientable and belongs to one of the following categories. %(see Section \ref{decomposition} for more detail):
\begin{enumerate}[(i)]
\item each $M_i$ is a closed prime $3$-manifold with $|\pi_1(M_i)|=\infty$ and is not an $S^2$- or $RP^2$-bundle over $S^1$;
\item  each $N_j$ is a closed prime $3$-manifold with $|\pi_1(N_j)|<\infty$;
\item  each $Q_l$ is a prime $3$-manifold and $\partial Q_l$ is non-empty;
\item  each $S_e$ is an $S^2$- or $RP^2$-bundle over $S^1$.  
\end{enumerate}
Note in (ii) each $N_j$  is orientable, and in (iv) each $RP^2$-bundle over $S^1$ is homeomorphic to $RP^2\times S^1$.
Also each orientable 3-manifold contains no embedded 2-sided $RP^2$, and each 3-manifold containing embedded $RP^2$ has 2-torsion in its fundamental group.

We make the following conjecture about $\chi_4(\pi)$ and  $p(\pi)$.

\begin{conjecture}\label{conj}
Suppose $M$ is a compact 3-manifold with prime decomposition
described  as in (\ref{1.2}) and (i)-(iv).
Let $\pi= \pi_1(M)$, then 
$$ \chi_4(\pi)=  p(\pi)=2-2(p+q)+\chi(\partial M).$$
\end{conjecture}

If Conjecture \ref{conj} is true, it implies that any closed orientable 4-manifold $X$ realizing $\chi_4(\pi)$ has zero signature.

We state our works on $\chi_4(\pi)$ and related invariants in the following three theorems.
Theorem \ref{main1.2} is the main result of this paper and is more difficult to prove.  
Theorems \ref{main1.2},  \ref{main1.3} and  Remark \ref{last} (1)  support  Conjecture \ref{conj}.

\begin{theorem}\label{main1.2}
Suppose $M$ is a compact 3-manifold with prime decomposition
described  as in (\ref{1.2}) and (i)-(iv).
Let $\pi= \pi_1(M)$, if each $M_i$ in (i) and each $Q_l$ in (iii)  contains no two-sided $RP^2$, then
\begin{align}\label{1.3}
\chi_4(\pi)=  2-2(p+q)+\chi(\partial M).
\end{align}

In particular, (\ref{1.3}) holds when $M$ contains no two-sided $RP^2$.
\end{theorem}

\begin{theorem}\label{main1.3}
Suppose $M$ is a compact 3-manifold with prime decomposition
described  as in (\ref{1.2}) and (i)-(iv).
Let $\pi= \pi_1(M)$, if each closed 3-manifold $M_i$ in (i) is orientable, then 
\begin{align}\label{1.4}
p(\pi) = \chi_4(\pi)=  2-2(p+q)+\chi(\partial M).
\end{align}

In particular, (\ref{1.4}) holds when $M$ is orientable.
\end{theorem}
 
The third result confirms the question asked by Hillman above, and in fact we prove a stronger result.
\begin{theorem}\label{main1.4}
Suppose $M$ is a compact 3-manifold with prime decomposition
described  as in (\ref{1.2}) and (i)-(iv). Let $\pi=\pi_1(M)$, if $\pi$ contains no $2$-torsions, 
then 
\begin{align}\label{1.5}
q^*(\pi)=  \chi_4(\pi)=  2-2(p+q)+\chi(\partial M).
\end{align}
In particular, if $p=q=0$, then
$q^*(\pi)=2.$
\end{theorem}

\begin{remark}\label{last} (1) Let $M$ be a closed, irreducible and non-orientable 3-manifold and  $M\neq RP^2\times S^1$.
Let $r$ the number of disjoint non-parallel 2-sided $RP^2$'s in $M$. Then $\chi_4(\pi_1(M))=2$ if $r=0$ by Theorem \ref{main1.2}, and
\begin{align}\label{1.6}
\chi_4(\pi_1(M))\in \{0, 1, 2\} \  \ \text{if} \ \ r>0 
\end{align}
by the proof of Theorem \ref{main1.3} (Proposition \ref{b}) since $p=q=\chi(\partial M)=0$ and the number of non-orientable summands $k$ is $1$ in this case. Equation (\ref{1.6}) suggests  that $\chi_4(\pi_1(M))$ should be independent of the number 
of 2-sided $RP^2$'s in $M$, and to determine whether $\chi_4(\pi_1(M))=2$ holds is the first step to verify the conjecture.

(2) By Theorem \ref{main1.4}, if  $\pi_1(M)$ contains no $2$-torsion, then $q^*(\pi)$ can be realized by closed orientable 4-manifolds. 
Note $q^*(\mathbb{Z}_2)=1$ is realized by $RP^4$. So the equality $ q^*(\pi)=  \chi_4(\pi)$ in Theorem \ref{main1.4} is not true for all 3-manifold groups (since $\chi_4(\mathbb{Z}_2)=2$ by Theorem \ref{main1.2}). 
\end{remark} 

The organization  of the paper is reflected by the table of the content.
%as following:  In Section 2, Proposition \ref{main1} provides the upper bound of $\chi_4(\pi)$.
%Theorems  \ref{main1.2}, \ref{main1.3} and \ref{main1.4} are  proved in Sections \ref{secondhalf2},  \ref{secondhalf1} and \ref{wo2torsion} respectively.
%The Appendix  calculates   the Betti numbers of compact $3$-manifold groups. 
%In Section \ref{intersectionform}, we study intersection forms of closed orientable 4-manifolds realizing $\chi_4(\pi)$ for a 3-manifold group $\pi$.

\section{An upper bound of $\chi_4(\pi)$ and computations of $\beta_i(\pi)$}\label{firsthalf}

We first list several standard facts which will be repeatedly used in our proofs.%\cite[Exercises 20, 21, 22, p157]{Ha1}:

\begin{lemma}\label{standard} Suppose  $N_1$ and $N_2$ are compact n-manifolds.

(i) $\pi_1(N_1\# N_2)=\pi_1(N_1)*\pi_1(N_2)$ for $n\ge 3$, 

%(ii) If $X$ and $Y$ are finite CW complexes, then $\chi(X\times Y)=\chi(X)\times \chi(Y)$.

%(iii) If a finite CW complex $X$ is a union of sub-complexes $A$ and $B$, then  $\chi(X)=\chi(A)+\chi(B)-\chi(A\cap B)$. 

(ii) $\chi(N_1\#N_1)=\chi(N_1)+\chi(N_2)-2$ for $n=4$

(iii) If $p: N_1\to N_2$ is a covering map of degree $p$, then $\chi(N_1)=p\chi(N_2)$.
\end{lemma}

An upper bound of $\chi_4(\pi)$ in Theorems \ref{main1.2}, \ref{main1.3},  \ref{main1.4} is given in the following.

\begin{proposition}\label{main1}
Suppose $M$ is a compact 3-manifold with a prime decomposition
described  as in (\ref{1.2}) and (i)-(iv).
Let $\pi= \pi_1(M)$, then 
$$\chi_4(\pi)\le 2-2(p+q)+\chi(\partial M).$$ 
\end{proposition}

For a compact 3-manifold $M$
associated with decomposition (\ref{1.2}) and (i)-(iv),  to prove Proposition \ref{main1}, we only need to construct a closed orientable 4-manifold $M^*$   with $\pi_1(M^*)=  \pi_1(M)$ and 
$$\chi(M^*)=2-2(p+q)+\chi(\partial M).$$

{\it Construction  of $M^*$:} Suppose $M$ is a compact 3-manifold and $\partial M$ has $k$ components $\{S_1, ..., S_k\}$.

Case 1: $M$ is orientable and $\partial M\ne \emptyset$.  Let 
$$M^{*}=(M\times S^1) \cup (\cup_i S_i\times D^2),$$
where each component $S_i\times S^1$ of $\partial (M\times S^1)$ is identified with $\partial (S_i\times D^2)=S_i\times S^1$ canonically.
By Van Kampen theorem, we can verify that $\pi_1(M^{*})=\pi_1(M).$

Case  2: $M$ is   non-orientable and $\partial M\ne \emptyset$.  
Let $p: \tilde M\to M$ be the orientable double cover of $M$, with a fixed-point free orientation reversing involution
$\tau : \tilde M\to \tilde M$ such that $\tilde M/\tau=M$.
Let $r: S^1\to S^1$ be an orientation reversing involution on $S^1$.
Now we have an orientation preserving fixed-point free involution $\tau\times r: \tilde M\times S^1\to \tilde M\times S^1.$
Then $\tilde M\times S^1/\tau\times r$ is an orientable 4-manifold, which indeed is a twisted product $M\tilde \times S^1$.
Each boundary component of $M\tilde \times S^1$ is either  $S_j\times S^1$ if $S_j$ is orientable, or $S_j\tilde \times S^1$ otherwise. Then we close these components canonically by $S_j\times D^2$ or $S_j\tilde \times D^2$, depending on $S_j$ is orientable or not.
Again we get a closed orientable 4-manifold $M^{*} $ with $\pi_1(M^{*})=\pi_1(M).$

Case 3: $M$ is closed. let $B^3$ be a 3-ball in $M$, we denote by $\breve M=M\setminus \text{int} B^3$.
Then $\partial \breve M=S^2$, and  $\breve M^*$ is defined.

{\it Final Construction:} For a compact 3-manifold $M$ with prime decomposition given in (\ref{1.2}), denote the connected sum of closed 3-manifold pieces in (i) and (ii) by $P=(\#_{i=1}^mM_i)\#(\#_{j=1}^n N_j)$. We suppose that there are $q_1$ prime factors that are $S^2$-bundles and there are $q_2$ prime factors that are $RP^2$-bundles, with $q_1+q_2=q$. Then we define 
\begin{align}\label{2.1}
M^*= \breve {P}^{*}\# (\#_{l=1}^p Q_l^{*})\#(\#_{e=1}^{q_1}(S^3\times S^1))\#(\#_{f=1}^{q_2}(RP^3\times S^1)),
\end{align}
which is a closed orientable 4-manifold. 
  
By Lemma \ref{standard} it is easily shown  that
$$\pi_1(M^*)\cong \pi_1(M)\ \ \text{and} \ \  \chi(M^*)= 2-2(p+q)+\chi(\partial M).$$
Then Proposition \ref{main1} follows.

The proofs of Theorems \ref{main1.2}, \ref{main1.3} and \ref{main1.4} use the following proposition. 

\begin{proposition}\label{kpi1}
Let $M$ be a compact $3$-manifold, and let its prime decomposition be
$$M=(\#_{i=1}^mM_i)\#(\#_{j=1}^n N_j)\#(\#_{l=1}^p Q_l)\#(\#_{e=1}^q S_e)$$ as in (\ref{1.2}) and (i)-(iv).
%Here each $M_i$ is a closed prime $3$-manifold with $|\pi(M_i)|=\infty$ and is not an $S^2$- or $RP^2$-bundle, each $N_j$ is closed and has finite fundamental group,
%each $Q_l$ is prime and has non-empty  $\partial Q_l$, each $S_e$ is an $S^2$- or $RP^2$-bundle over $S^1$.  
Let $\pi= \pi_1(M)$,  then we have

(1) $\beta_4(\pi) =0$;

(2) $\beta_2(\pi) =\Sigma_{i=1}^m \beta_2(M_i)+\Sigma_{l=1}^p \beta_2(Q_l).$

Moreover, if $\pi$ contains no $2$-torsion, then 
 
(3) $\beta_4(\pi; \mathbb{Z}_2) =0$;

(4) $\beta_2(\pi; \mathbb{Z}_2) =\Sigma_{i=1}^m \beta_2(M_i; \mathbb{Z}_2)+\Sigma_{l=1}^p \beta_2(Q_l; \mathbb{Z}_2).$
\end{proposition}

The following result is used to prove Proposition \ref{kpi1} and Proposition \ref{Lefschetz}.

\begin{proposition}\label{Hatcher}  \cite[Proposition 3G.1]{Ha1}
Let $p: \tilde X \to X$ be a finite regular covering with deck transformation group $\Gamma$. Then

(1) $p^*: H^k(X)\to H^k(\tilde X)$ is injective for each integer $k$.

(2) the image of  $p^*$  is the subspace $H^k(\tilde X)^\Gamma$ consisting of elements fixed by all $\gamma\in \Gamma$.

\end{proposition}

To prove Proposition \ref{kpi1}, we will first work on prime 3-manifolds. 
\begin{proposition}\label{kpi1prime}
Let $M$ be a prime $3$-manifold that is not an $S^2$- or $RP^2$-bundle over $S^1$, and let $\pi=\pi_1(M)$. Then we have
$\beta_2(\pi)=\beta_2(M)$.
\end{proposition}

\begin{proof}
{\bf Case I.} $M$ is an orientable $3$-manifold. 

If $M$ has finite fundamental group, since no boundary component of $M$ is $S^2$, $M$ must be closed. Then the result holds since $\beta_2(M)=0$ (by Poincare duality) and any finite group has zero betti numbers in all positive dimensions (for example, see Theorem 6.5.8 of \cite{We}).

If $M$ has infinite fundamental group, since $M$ is not an $S^2$-bundle, it is an irreducible $3$-manifold. By the sphere theorem and the Hurewicz theorem, $M$ is aspherical and it is a model of $K(\pi,1)$. So %$\beta_4(K(\pi,1)) =\beta_4(M)=0$ and 
$\beta_2(\pi)=\beta_2(M)$ holds.

{\bf Case II.} $M$ is a non-orientable $3$-manifold.

Let $i:M \to X=K(\pi,1)$ be the inclusion that induces an isomorphism on $\pi_1$. 
Let $p_M: \tilde M\to M$ be the orientable double cover with $\pi_1(\tilde M) =\tilde \pi< \pi$, 
 and $p_X: \tilde X=K(\tilde{\pi},1)\to X=K(\pi,1)$ be the double cover corresponding to $\tilde{\pi }<\pi$. We get the following commutative diagrams:
\begin{diagram}
\tilde{M} & \rTo^{\tilde i} & \tilde X  \ \  \ \  \ \ \tilde{M} & \rTo^{\tau_M} & \tilde M                     \\     
\dTo^{p_M} & & \dTo^{p_X} \ \  \ \  \ \ \dTo^{\tilde i} & & \dTo^{\tilde i} \\
M& \rTo^i & X, \ \  \ \ \ \  \tilde X& \rTo^{\tau_X} & \tilde X,
\end{diagram}
where  $\tau_M:  \tilde M\to \tilde M$ and $\tau_X: \tilde X\to \tilde X$ are nontrivial deck transformations
and $\tilde i: \tilde M \to \tilde X$ is an inclusion.

%By Proposition \ref{Hatcher},  we have

%$$\beta_2(M)=\text{dim}H^2(\tilde M)^{\tau_M^*}, \ \  \beta_2(X)=\text{dim}H^2(\tilde X)^{\tau_X^*}.$$

Let $\mathbb{P}$ be a maximal collection of disjoint non-parallel two-sided $RP^2$ in $M$ (such a collection exists by an exercise in \cite{Ha2} page 12). Let $\mathbb{S}$ be the preimage of $\mathbb{P}$ in $\tilde{M}$, then each component of $\mathbb{S}$ is a $2$-sphere. Moreover, for each component $K$ of $\tilde{M}\setminus \mathbb{S}$, if we cap off each $S^2$ boundary component of $K$ by a $3$-ball to obtain $\hat{K}$, a classical argument in $3$-manifold topology implies that $\hat{K}$ is irreducible. So if we pinch each component of $\mathbb{S}$ to a point, we get a space $\tilde{X}^{(3)}$ homotopic equivalent to a one-point union of orientable irreducible $3$-manifolds and $S^1$. Then we can add cells of dimension at least $4$ to $\tilde{X}^{(3)}$ to obtain $\tilde{X}$, which is a model of $K(\tilde{\pi},1)$. Since adding cells of dimension at least $4$ does not affect $H^2$, we have an exact sequence
$$0=H^1(\mathbb{S})\to H^2(\tilde{X})=H^2(\tilde{X}^{(3)})\to H^2(\tilde{M})\to H^2(\mathbb{S})$$
which is $\mathbb Z_2$-equivariant (induced by the action $\tau_X$). Since $\mathbb Z_2$
acts on $H^2(\mathbb S)$ as multiplying by $-1$, by applying Proposition \ref{Hatcher} twice, we have
$$\beta_2(\pi)=\beta_2(X)=\text{dim}H^2(\tilde X)^{\tau_X^*}=\text{dim}H^2(\tilde M)^{\tau_M^*}=\beta_2(M).$$

%%By Case I, we have $\beta_4(\tilde X)=\beta_4(K(\tilde \pi,1)) =0$. Then  $\beta_4(X)=\beta_4(K( \pi,1)) =0$ by Proposition \ref{Hatcher} (1).

%Note $\tilde i^*:  H^2(\tilde X) \to H^2 (\tilde M)$ is injective. 
%Since  $\beta_2(K(\tilde \pi,1))=\beta_2(\tilde M)$ by Case I, we have that  $\tilde i^*:  H^2(\tilde X) \to H^2 (\tilde M)$ is an isomorphism.
%So the involutions $\tau^*_M$ and $\tau^*_X$ are conjugated by $\tilde i^*$.  Thus
%$$\tilde i^*(H^2(\tilde X)^{{\tau_X^*}})=H^2(\tilde M)^{\tau_M^*}.$$
% By Proposition \ref{Hatcher} (2), we have $$\tilde i^*(im p^*_X)=im p^*_M.$$
%Since $\tilde i^*$ is an isomorphism, $im p^*_X$ and $ im p^*_M$ have the same dimension.
%Since $p^*_X$ and $p^*_M$ are injective by Proposition \ref{Hatcher} (1), $H^2(X)$ and $H^2(M)$ have the same dimension, 
%that is $$\beta_2(G)=\beta_2(X)=\beta_2(M).$$
\end{proof}

% Then by  $\tilde i\circ \tau_ M=\tau_ X \circ \tilde i$ 

\begin{proof}[Proof of Proposition \ref{kpi1}]
Suppose the $3$-manifold $M$ has prime decomposition 
$$M=(\#_{i=1}^mM_i)\#(\#_{j=1}^n N_j)\#(\#_{l=1}^p Q_l)\#(\#_{e=1}^q S_e)$$ 
as in (\ref{1.2}).
Here each $M_i$ is a closed prime $3$-manifold with $|\pi_1(M_i)|=\infty$ and is not an $S^2$- or $RP^2$-bundle, each $N_j$ is closed and has finite fundamental group,
each $Q_l$ is prime and has non-empty  $\partial Q_l$, each $S_e$ is an $S^2$- or $RP^2$-bundle over $S^1$.  

Among the $q$ prime factors of $M$ that are $S^2$- or $RP^2$-bundles, we suppose $q_1$ of them are $S^2$-bundles and $q_2$ of them are $RP^2$-bundles. Then a $K(\pi_1(M),1)$ space can be taken to be 
$$(\vee_{i=1}^mK(\pi_1(M_i),1))\vee (\vee_{j=1}^nK(\pi_1(N_j),1))\vee (\vee_{l=1}^pK(\pi_1(Q_l),1))\vee$$
$$ (\vee^{q_1}S^1)\vee(\vee^{q_2}RP^{\infty}\times S^1).$$

We  have 
\begin{align*}
&\beta_2(\pi_1(M))=\sum_{i=1}^m\beta_2(\pi_1(M_i))+\sum_{j=1}^n\beta_2(\pi_1(N_j))+\sum_{l=1}^p\beta_2(\pi_1(Q_l))\\
=\ &\sum_{i=1}^m\beta_2(M_i)+\sum_{j=1}^n\beta_2(N_j)+\sum_{l=1}^p\beta_2(Q_l)=\sum_{i=1}^m\beta_2(M_i)+\sum_{l=1}^p\beta_2(Q_l).
\end{align*}
Here the second equality follows from Proposition \ref{kpi1prime}  and the third equality follows from the fact that $\pi_1(N_j)$ is finite.

We always can find an orientable  finite covering $\tilde M$ of $M$ such that each prime factor $\tilde M$ is either aspherical 
or $S^2\times S^1$ or has finite $\pi_1$. Then by the argument we used above, it is easy to see  $\beta_4(\pi_1(\tilde M))=0$.
Then by Proposition \ref{Hatcher} (1), $\beta_4(\pi_1( M))=0$. We have proved  (1) and  (2) of Proposition \ref{kpi1}.

Suppose that $\pi$ contains no $2$-torsions.  We conclude that each $M_i$ and $Q_l$ contains no 2-sided projective plane, 
 each $S_e$ is an $S^2$-bundle over $S^1$, and the fundamental group of each $N_j$ has odd order. As we discussed in the proof of Proposition \ref{kpi1prime}, each 
 $M_i$ and $Q_l$ is aspherical and 
$$K(\pi_1(M),1)=(\vee_{i=1}^m M_i)\vee (\vee_{j=1}^n K(\pi_1(N_j),1))\vee(\vee_{l=1}^p Q_l)\vee (\vee^{q}S^1).$$

Since $\pi_1(N_j)$ is odd, $\beta_k(\pi_1(N_j);\mathbb{Z}_2)=0$ for all $k\geq 1$ (for example, see Theorem 6.5.8 of \cite{We}). Then (3) and (4) of Proposition \ref{kpi1} follows.
\end{proof}

\section{Obstructions for  lower bounds of $\chi_4(\pi)$, $p(\pi)$ and $q^*(\pi)$}\label{secondhalf2}

Propositions \ref{Lefschetz}, \ref{main2} and \ref{euler},  the three obstruction results in this section,  are  for  lower bounds of $\chi_4(\pi)$, $p(\pi)$ and $q^*(\pi)$ respectively.  Both Proposition \ref{Lefschetz} and its proof are new, Proposition  \ref{main2} is known,
and both Proposition \ref{euler} and its proof are $\mathbb{Z}_2$-version of known results.

\begin{proposition}\label{Lefschetz}
Let $G$ be a finitely presented group, and let $\tilde{G}<G$ be an index-$2$ subgroup such that $\beta_4(\tilde{G})=0$. Then we have 
$$\chi_4(G)\geq 1-\beta_1(\tilde{G})+\beta_2(\tilde{G})+|2(\beta_1(G)-\beta_2(G))-(\beta_1(\tilde{G})-\beta_2(\tilde{G}))-1|.$$
\end{proposition}

\begin{proof}
Let $X$ be a closed orientable $4$-manifold with $\pi_1(X)\cong G$, then we have a map $i:X\to K(G,1)$ that induces an isomorphism on $\pi_1$. (Here we use the fact that any compact manifold is homotopic equivalent to a CW-complex, see \cite[ Corollary A.12]{Ha1}.)

Let $p:\tilde{X}\to X$ and $q:K(\tilde{G},1)\to K(G,1)$ be the double covers of $X$ and $K(G,1)$ corresponding to $\tilde{G}<G$ respectively, we get the following commutative diagram:
\begin{diagram}
\tilde{X} & \rTo^j & K(\tilde{G},1)\\
\dTo^p & & \dTo^q\\
X& \rTo^i & K(G,1)
\end{diagram}

Since $\pi_1(\tilde{X})\cong \tilde{G}$, we have $\beta_1(\tilde{X})=\beta_1(\tilde{G})$. By Poincare duality, we have $\beta_3(\tilde{X})=\beta_1(\tilde{X})=\beta_1(\tilde{G})$. So we get 
\begin{align}\label{3.1}
\chi(\tilde{X})=2-2\beta_1(\tilde{G})+\beta_2(\tilde{X}),\ \chi(X)=1-\beta_1(\tilde{G})+\frac{1}{2}\beta_2(\tilde{X}).
\end{align}
To prove this proposition, we only need to bound $\beta_2(\tilde{X})$ from below.

Since $j:\tilde{X}\to K(\tilde{G},1)$ induces an isomorphism on $\pi_1$, $j^*:H^1(K(\tilde{G},1))\to H^1(\tilde{X})$ is an isomorphism. Since $K(\tilde{G},1)$ can be obtained  (up to homotopy equivalence) by attaching cells to $\tilde{X}$ with dimension at least $3$,  $j^*:H^2(K(\tilde{G},1))\to H^2(\tilde{X})$ is injective.

Let $\tau_X:\tilde{X}\to \tilde{X}$ and $\tau_K:K(\tilde{G},1)\to K(\tilde{G},1)$ be the nontrivial deck transformations of $p:\tilde{X}\to X$ and $q:K(\tilde{G},1)\to K(G,1)$ respectively. 
%We also have a commutatie diagram:
%\begin{diagram}
%\tilde{X} & \rTo^j & K(\tilde{G},1)\\
%\dTo^{\tau_X} & & \dTo^{\tau_K}\\
%\tilde{X}& \rTo^j & K(\tilde{G},1)
%\end{diagram}
%Since $q:K(\tilde{G},1)\to K(G,1)$ is a double cover, 
By Proposition \ref{Hatcher}, for each $n$, $q^*:H^n(K(G,1))\to H^n(K(\tilde{G},1))$ is injective, and we have $$q^*(H^n(K(G,1)))=(H^n(K(\tilde{G},1)))^{\tau_K^*}.$$
%Here $(H^n(K(\tilde{G},1)))^{\tau_K^*}$ denotes the subspace of $H^n(K(\tilde{G},1))$ consisting of elements fixed by $\tau_K^*:H^n(K(\tilde{G},1))\to H^n(K(\tilde{G},1))$.

Since $\tau_K^*:H^n(K(\tilde{G},1))\to H^n(K(\tilde{G},1))$ gives a $\mathbb{Z}_2$-action, we have 
$$H^n(K(\tilde{G},1))=H^n(K(\tilde{G},1))_+\oplus H^n(K(\tilde{G},1))_-,$$
where $H^n(K(\tilde{G},1))_+$ and $H^n(K(\tilde{G},1))_-$ denote the eigenspaces of $\tau_K^*$ corresponding to eigenvalues $1$ and $-1$ respectively. Similarly, by considering the $\mathbb{Z}_2$-action on $H^n(\tilde{X})$ given by $\tau_X^*$ and the eigenspaces corresponding to $1$ and $-1$, we have 
$$H^n(\tilde{X})=H^n(\tilde{X})_+\oplus H^n(\tilde{X})_-.$$

Then we have $$H^n(K(\tilde{G},1))_+= H^n(K(\tilde{G},1))^{\tau_K^*}=q^*(H^n(K(G,1))).$$ 
Since $q^*:H^n(K(G,1))\to H^n(K(\tilde{G},1))$ is injective, we have 
$$\text{dim}\ H^n(K(\tilde{G},1))_+=\text{dim}\ q^*(H^n(K(G,1)))=\text{dim}\ H^n(K(G,1)=\beta_n(G)$$ 
and 
$$\text{dim}\ H^n(K(\tilde{G},1))_-=\text{dim}\ H^n(K(\tilde{G},1))-\text{dim}\ H^n(K(\tilde{G},1))_+=\beta_n(\tilde{G})-\beta_n(G).$$

Since $j^*:H^1(K(\tilde{G},1))\to H^1(\tilde{X})$ is an isomorphism and commutes with the action of deck-transformation, we have 
$$\text{dim}\ H^1(\tilde{X})_+=\text{dim}\ j^*(H^1(K(\tilde{G},1))_+)=\text{dim}\ H^1(K(\tilde{G},1))_+=\beta_1(G)$$
and 
$$\text{dim}\ H^1(\tilde{X})_-=\text{dim}\ j^*(H^1(K(\tilde{G},1))_-)=\text{dim}\ H^1(K(\tilde{G},1))_-=\beta_1(\tilde{G})-\beta_1(G).$$ So we get 
\begin{align}\label{3.2}
\text{tr}(\tau_X^*:H^1(\tilde{X})\to H^1(\tilde{X}))=2\beta_1(G)-\beta_1(\tilde{G}).
\end{align}

Since $\tau_X:\tilde{X}\to \tilde{X}$ is an orientation preserving homeomorphism, for any $\alpha\in H^1(\tilde{X})$ and $\beta\in H^3(\tilde{X})$, we have $\alpha\cup \beta=\tau_X^*(\alpha)\cup \tau_X^*(\beta)\in H^4(\tilde{X})\cong \mathbb{R}$. By Poincare duality, we have $\text{dim}\ H^3(\tilde{X})_+=\text{dim}\ H^1(\tilde{X})_+=\beta_1(G)$ and  $\text{dim}\ H^3(\tilde{X})_-=\text{dim}\ H^1(\tilde{X})_-=\beta_1(\tilde{G})-\beta_1(G)$, so we get
\begin{align}\label{3.3}
\text{tr}(\tau_X^*:H^3(\tilde{X})\to H^3(\tilde{X}))=2\beta_1(G)-\beta_1(\tilde{G}).
\end{align}

By Poincare duality, the cup product $H^2(X)\times H^2(X)\to H^4(X)\cong \mathbb{R}$ is a non-singular bilinear form. 

Now we need two lemmas and the first  one is well-known.

\begin{lemma}\label{algebra}
Let $V$ be a vector space of dimension $n$ over a field $F$ and let $q$ be a non-degenerate symmetric bilinear form on $V$.
 If $q$ vanishes on a sub-space $W$ of dimension $m$, then $n\ge 2m$. 
\end{lemma}

\begin{lemma}\label{+-}
(1) The restrictions of the cup product of $H^*(\tilde{X})$ on both $H^2(\tilde{X})_+\times H^2(\tilde{X})_+$ and  $H^2(\tilde{X})_-\times H^2(\tilde{X})_-$ are non-degenerate. 

(2) The restrictions of the cup product of $H^*(\tilde{X})$ on both $j^*(H^2(K(\tilde{G},1))_+)\times j^*(H^2(K(\tilde{G},1))_+)$ and 
$j^*(H^2(K(\tilde{G},1))_-)\times j^*(H^2(K(\tilde{G},1))_-)$ are trivial.

 \end{lemma}

\begin{proof} %Since the cup product $H^2(X)\times H^2(X)\to H^4(X)\cong \mathbb{R}$ is a nonsingular bilinear form. 
%$H^2(\tilde{X})_+\times H^2(\tilde{X})_+$ 

(1) For any non-zero elements  $\alpha \in H^2(\tilde{X})_+$ and $\beta \in H^2(\tilde{X})_-$, then $\tau_X^*(\alpha)=\alpha$
and $\tau_X^*(\beta)=-\beta$, hence $$\tau_X^*(\alpha\cup \beta)=-\alpha\cup \beta.$$
However since $\tau_X$ is orientation preserving, we have  
$$\tau_X^*(\alpha\cup \beta)=\alpha\cup \beta$$
which implies that $\alpha\cup \beta=0$.

Since the cup product $H^2(\tilde{X})\times H^2(\tilde{X})\to H^4(\tilde{X})\cong \mathbb{R}$ is a non-singular bilinear form, for any nonzero $\alpha\in H^2(\tilde{X})_+$, there must be an non-zero element $\gamma \in H^2(\tilde{X})_+$ such that $\gamma\cup \alpha\ne 0$. So 
the restriction of the cup product of 
$H^*(\tilde{X})$ on $H^2(\tilde{X})_+\times H^2(\tilde{X})_+$ is non-degenerate. 
The same argument works for $H^2(\tilde{X})_-$. 

(2) Since $\beta_4(\tilde{G})=0$, the restrictions of the cup product of $H^*(K(\tilde{G},1))$ on $H^2(K(\tilde{G},1))_+\times H^2(K(\tilde{G},1))_+$ and $H^2(K(\tilde{G},1))_-\times H^2(K(\tilde{G},1))_+$ are trivial. Since $j^*(H^2(K(\tilde{G},1))_+)\subset H^2(\tilde{X})_+$ and
$j^*(H^2(K(\tilde{G},1))_-)\subset H^2(\tilde{X})_-$, the conclusion follows.
\end{proof}

By Lemmas \ref{algebra} and \ref{+-},
we have 
\begin{align}\label{3.4}
\text{dim}\ H^2(\tilde{X})_+\geq 2\ \text{dim}\ j^* (H^2(K(\tilde{G},1))_+) =2\ \text{dim}\ H^2(K(\tilde{G},1))_+=2\beta_2(G),
\end{align}
and
\begin{align}\label{3.5}
\text{dim}\ H^2(\tilde{X})_-\geq 2\  \text{dim}\ H^2(K(\tilde{G},1))_-=2\beta_2(\tilde{G})-2\beta_2(G)
\end{align}
The above inequalities imply that 
\begin{align}\label{3.6}
\text{dim}\ H^2(\tilde{X})_+=2\beta_2(G)+\Delta_+,\ \text{dim}\ H^2(\tilde{X})_-=2\beta_2(\tilde{G})-2\beta_2(G)+\Delta_-.
\end{align}

for some non-negative integers $\Delta_+,\Delta_-$.
So we have
\begin{align}\label{3.7}
\text{tr}(\tau_X^*:H^2(\tilde{X})\to H^2(\tilde{X}))=4\beta_2(G)-2\beta_2(\tilde{G})+(\Delta_+-\Delta_-)
\end{align}

Since $\tau_X:\tilde{X}\to \tilde{X}$ has no fixed-point, by the Lefschetz fixed-point theorem and equations (\ref{3.2}), (\ref{3.3}), (\ref{3.7}), we have 
\begin{align*}
&0=\sum_{i=0}^4(-1)^i\cdot \text{tr}(\tau_X^*:H^i(\tilde{X})\to H^i(\tilde{X}))\\
=\ & 1-(2\beta_1(G)-\beta_1(\tilde{G}))+(4\beta_2(G)-2\beta_2(\tilde{G})+(\Delta_+-\Delta_-))-(2\beta_1(G)-\beta_1(\tilde{G}))+1\\
=\ & 2-4(\beta_1(G)-\beta_2(G))+2(\beta_1(\tilde{G})-\beta_2(\tilde{G}))+(\Delta_+-\Delta_-).
\end{align*}
Thus 
\begin{align}\label{3.8}
\Delta_+-\Delta_-=4(\beta_1(G)-\beta_2(G))-2(\beta_1(\tilde{G})-\beta_2(\tilde{G}))-2.
\end{align}

By equation (\ref{3.1}), we have 
\begin{align*}
&\chi(X)=1-\beta_1(\tilde{G})+\frac{1}{2}(\text{dim}\ H^2(\tilde{X})_++\text{dim}\ H^2(\tilde{X})_-)\\
=\ & 1-\beta_1(\tilde{G})+\beta_2(\tilde{G})+\frac{1}{2}(\Delta_++\Delta_-)  \qquad \text{by\ (\ref{3.6})}\\
\geq \ & 1-\beta_1(\tilde{G})+\beta_2(\tilde{G})+\frac{1}{2}|\Delta_+-\Delta_-|\\
=\ & 1-\beta_1(\tilde{G})+\beta_2(\tilde{G})+|2(\beta_1(G)-\beta_2(G))-(\beta_1(\tilde{G})-\beta_2(\tilde{G}))-1| \qquad \text{by\ (\ref{3.8})}.
\end{align*}

\end{proof}

\begin{proposition}\label{main2}\cite[Theorem 4.2]{Ko1}
If $\beta_4(\pi)=0$, then
$$p(\pi)\ge  2-2\beta_1(\pi)+2\beta_2(\pi).$$
\end{proposition}

\begin{proposition}\label{euler}
Let $M$ be a compact $3$-manifold whose fundamental group $\pi$ containing no $2$-torsion, and let $X$ be a closed $4$-manifold with $\pi_1(X)=\pi$. Then we have $$\beta_2(X; \mathbb{Z}_2)\geq 2\beta_2(\pi; \mathbb{Z}_2).$$
\end{proposition}

Since the group $\pi$ of $M$  contains no $2$-torsion,  $H^4(\pi; \mathbb{Z}_2)$ is trivial by Proposition \ref{kpi1} (3). 
Then the proof  of Proposition \ref{euler} is analogue to that of that of Proposition \ref{main2} (or see directly the proof of inequality (\ref{3.4}) above)  by applying Lemma \ref{algebra} for $F=\mathbb{Z}_2$.

\section{Proofs of theorems}

We will introduce some notions and  conventions to simplify computations in our proofs. 
For each finitely presented group $G$, define
$$b(G)=\beta_2(G)-\beta_1(G), \ \ b(G; \mathbb{Z}_2)=\beta_2(G; \mathbb{Z}_2)-\beta_1(G; \mathbb{Z}_2).$$
The verification of the following fact is routine.
\begin{lemma}\label{additive}
For finitely presented group $G_1$ and $G_2$, $$b(G_1*G_2)=b(G_1)+b(G_2), \ \ b(G_1*G_2; \mathbb{Z}_2)=b(G_1; \mathbb{Z}_2)+b(G_2;  \mathbb{Z}_2).$$
\end{lemma}

{\bf  Convention(*)} Suppose $M$ is a compact 3-manifold with $\pi= \pi_1(M)$ and with prime decomposition
described  as in (\ref{1.2}) and (i)-(iv).

Note that capping each $S^2$ boundary component of $\partial M$ by a 3-ball and swapping each $S^2$-bundle factor by $D^2\times S^1$
do not affect either the group or the right sides  of formulas (\ref{1.3}), (\ref{1.4}), (\ref{1.5})  in Theorems \ref{main1.2}, \ref{main1.3},
\ref{main1.4}. So below we assume that $M$ has neither $S^2$ boundary component nor  $S^2$-bundle factor.

\subsection{$p(\pi)=\chi_4(\pi)$ when closed prime factors  of $M$ are orientable}\label{secondhalf1}

\begin{proof}[Proof of Theorem \ref{main1.3}]
Suppose $M$ is a compact 3-manifold with prime decomposition
described  as in (\ref{1.2}) and (i)-(iv).
We have 
\begin{align}\label{4.1}
2-2(p+q)+\chi(\partial M)\ge \chi_4(\pi)\ge p(\pi)\ge 2-2\beta_1(\pi)+2\beta_2(\pi)=2+2b(\pi).
\end{align}
Here the first inequality follows from  Proposition \ref{main1}, the second inequality follows from (\ref{1.1}), the third inequality follows from
 Proposition \ref{main2} and Proposition \ref{kpi1} (1). 
 
 Theorem 1.3 follows from  (\ref{4.1}) and Proposition \ref{b} below, which give the explicit value of $2b(\pi)$ in term of presentation (\ref{1.2}), 
by taking $k=0$. 
\end{proof}
 
\begin{proposition}\label{b} Suppose $M$ is a compact 3-manifold with prime decomposition
described  as in (\ref{1.2}) and (i)-(iv) and let $\pi= \pi_1(M)$.  Suppose the number of closed  non-orientable 3-manifolds $M_i$ in (i) is $k$. Then 
$$2-2\beta_1(\pi)+2\beta_2(\pi)= 2-2(k+p+q)+\chi(\partial M).$$
\end{proposition}

\begin{proof} By Convention (*),  we can rewrite $M$ as $M=P\#N\#Q\#S$  satisfying the following:
\begin{enumerate}[(i)]
\item  $P=\#_{i=1}^{m+n-k}M_i$, each $M_i$ is closed orientable and and irreducible, with either finite or infinite fundamental group;
\item $W= \#_{j=1}^kW_j$ and each $W_j$ is a closed non-orientable $3$-manifold, $W_j\ne RP^2\times S^1$; 

\item $Q=\#_{l=1}^pQ_l$ and  each $Q_l$ is a prime $3$-manifold  with non-empty $\partial Q_l$;
\item  $S=\#_{e=1}^qS_e$ and each $S_e$ is homeomorphic to  $RP^2\times S^1$. 
\end{enumerate}

For  each $M_i$, we have 
\begin{align}\label{4.2}
b(\pi_1(M_i))=\beta_2(\pi_1(M_i))-\beta_1(\pi_1(M_i))
=\beta_2(M_i)-\beta_1(M_i)=0.
\end{align}
The second equality  follows from Proposition \ref{kpi1prime} and the last equality follows from Poincare duality for closed orientable 3-manifolds.

For  each $W_j$, we have 
\begin{align}\label{4.3}
b(\pi_1(W_j))=\beta_2(\pi_1(W_j))-\beta_1(\pi_1(W_j))=\beta_2(W_j)-\beta_1(W_j)=-1.
\end{align}
The second equality  follows from Proposition \ref{kpi1prime} and the last equality follows from Euler-Poincare formula for closed non-orientable 3-manifolds.

For  each $Q_l$, we have
\begin{align}\label{4.4}
b(\pi_1(Q_l))=\beta_2(\pi_1(Q_l))-\beta_1(\pi_1(Q_l)) =\beta_2(Q_l)-\beta_1(Q_l)
=\chi(Q_l)-1=\frac{\chi(\partial Q_l)}2-1.
\end{align}
The second equality  follows from Proposition \ref{kpi1prime}, the third  equality follows from Euler-Poincare formula for compact 3-manifold
with boundary.

Since $H^*(RP^\infty \times S^1)=H^*( S^1)$, we have 
\begin{align}\label{4.5}
b(\pi_1(RP^2 \times S^1))=-1.
\end{align}

Since  $b(*)$ is additive, by (\ref{4.2})-(\ref{4.5}), we have 
\begin{align*}
&b(\pi)=\Sigma_{i=1}^m b(\pi_1(M_i))+\Sigma_{j=1}^k b(\pi_1(W_j))+\Sigma_{l=1}^p b(\pi_1(Q_l))+\Sigma_{e=1}^q b(\pi_1(S_e))\\
=\  &\Sigma_{i=1}^m0+\Sigma_{j=1}^k(-1)+\Sigma_{l=1}^p(\frac{\chi(\partial Q_i)}2-1)+\Sigma_{e=1}^q(-1)= \frac{\chi(\partial M)}2-k- p-q.
\end{align*}

So 
$$2b(\pi)= -2(k+p+q)+\chi(\partial M).$$ 
\end{proof}

\subsection{$q^*(\pi)=\chi_4(\pi)$ for $\pi$ containing no $2$-torsions}\label{wo2torsion}

\begin{proof}[Proof of Theorem \ref{main1.4}]
%Suppose $M$ is a compact 3-manifold with prime decomposition
%described  as in (\ref{1.2}) and (i)-(iv) and  $\pi= \pi_1(M)$ contains no $2$-torsion.
By Proposition \ref{main1} and   (\ref{1.1}), we have
$$2-2(p+q)+\chi(\partial M)\ge \chi_4(\pi)\ge q^*(\pi)$$
To prove Theorem \ref{main1.4}, we need only to prove that 

\begin{align}\label{4.6}
q^*(\pi) \ge 2-2(p+q)+\chi(\partial M).
\end{align}

The condition that $\pi_1(M)$ contains no $2$-torsion implies that all prime factors of $M$ with finite fundamental groups are lens spaces $L(p, q)$ with odd $p$.  Moreover $M$ contain no 2-sided $RP^2$, so each  prime factor with infinite fundamental group is aspherical.
By Convention (*), we can rewrite $M$ as 
$$M=(\#_{i=1}^mM_i)\#(\#_{j=1}^nN_j)\#(\#_{l=1}^pQ_l)$$
such that 
each $M_i$ is a closed  aspherical $3$-manifold, each $N_j$ is a lens space with odd
order group,
and   each $Q_l$ is an aspherical $3$-manifold  with $\partial Q_l\ne \emptyset$. In particular $q=0$ holds here.
%\item  $S=\#_{e=1}^qS_e$, each $S_i$ is homeomorphic to an $S^2$-bundle over $S^1$. 

Using Poincare duality and  Euler-Poincare formula in $\mathbb{Z}_2$ coefficient and Proposition \ref{kpi1prime}, do the same computation 
as in the proof of Theorem \ref{main1.3},
we have 
\begin{align}\label{4.7}
b(\pi_1(M_i); \mathbb{Z}_2)=0.
\end{align}
and 
\begin{align}\label{4.8}
b(\pi_1(Q_l); \mathbb{Z}_2)=\chi(\partial Q_l)/2-1.
\end{align}
Since each $\pi_1(N_j)=Z_{k_j}$ with $k_j$ odd, we have $\tilde H^*(K(\pi_1(N_j), 1);  \mathbb{Z}_2)=0$, and it follows immediately that 
\begin{align}\label{4.9}
b(\pi_1(N_j); \mathbb{Z}_2)=0.
\end{align}

Since $b(*; \mathbb{Z}_2)$ is additive, by (\ref{4.7})-(\ref{4.9}) we have 
$$b(\pi; \mathbb{Z}_2)=\Sigma_{i=1}^m b(\pi_1(M_i); \mathbb{Z}_2)+\Sigma_{j=1}^n b(\pi_1(N_j); \mathbb{Z}_2)+\Sigma_{l=1}^p b(\pi_1(Q_l); \mathbb{Z}_2)$$ 
$$=\Sigma_{i=1}^m 0+\Sigma_{j=1}^n0+ \Sigma_{l=1}^p(\chi(\partial Q_i)/2-1)= \chi(\partial M)/2- p$$
So
$$\beta_2(\pi; \mathbb{Z}_2)=b(\pi; \mathbb{Z}_2)+ \beta_1(\pi;  \mathbb{Z}_2)=\chi(\partial M)/2- p+ \beta_1(\pi; \mathbb{Z}_2)$$

%The proof is similar to the proof of Proposition \ref{m1}, with real coefficient replaced by $\mathbb{Z}_2$-coefficients.
Suppose $X$ is a closed 4-manifold such that $\pi_1(X)\cong\pi$. 
By Proposition \ref{euler} and Proposition \ref{kpi1} (3), we have
$$\beta_2(X; \mathbb{Z}_2)\ge 2\beta_2(\pi; \mathbb{Z}_2)=\chi(\partial M)- 2p+ 2\beta_1(\pi; \mathbb{Z}_2).$$
Then
$$\chi(X)=2-2\beta_1(X; \mathbb{Z}_2)+\beta_2(X; \mathbb{Z}_2)=2-2\beta_1(\pi; \mathbb{Z}_2)+\beta_2(X; \mathbb{Z}_2)\ge 2-2p +\chi(\partial M)$$
So 
$$q^*(\pi)\ge  2-2p+\chi (\partial M).$$
 \end{proof}

\subsection{Proof of Theorem \ref{main1.2}}

\begin{proof}[Proof of Theorem \ref{main1.2}]
 We may assume that $M$ is non-orientable, otherwise it is proved in Theorem \ref{main1.3}.
Suppose $$M=P\# N,$$
where $P$ is orientable and each prime factor  $W$ of $N$ is non-orientable and contains no 2-sided $RP^2$
unless $W=RP^2\times S^1$.

Let $\tilde M$ be the orientable double covering of $M$, let $G=\pi_1(M)$ and $\tilde G=\pi_1(\tilde M)$.
By Proposition \ref{Lefschetz} and Proposition \ref{kpi1} (1), we have 
\begin{align}\label{4.10}
\begin{split}
&\chi_4(G)\geq 1-\beta_1(\tilde{G})+\beta_2(\tilde{G})+|2(\beta_1(G)-\beta_2(G))-(\beta_1(\tilde{G})-\beta_2(\tilde{G}))-1|\\
=\ & 1+b(\tilde G)+ |-2 b(G)+b(\tilde G)-1|.
\end{split}
\end{align}

Let $\tilde N$ be the orientable double covering of $N$, and let $\bar{P}$ be the orientation reversal of $P$. Then
$$\tilde M= P\# \tilde N \#\bar{P}$$
and
$$G=\pi_1(P)*\pi_1(N), \ \ \tilde G=\pi_1(P)*\pi_1(\tilde N)*\pi_1(P).$$
By the additivity of $b(*)$, we have 
$$b(G)=b(\pi_1(P))+ b(\pi_1(N)), \ \ b(\tilde G)=2b(\pi_1(P))+b(\pi_1(\tilde N)).$$
Substitute them in to (\ref{4.10}), we have 
\begin{align}\label{4.11}
\chi_4(G)\geq 1+2b(\pi_1(P))+b(\pi_1(\tilde N))+ |-2 b(\pi_1(N))+b(\pi_1(\tilde N))-1|. 
\end{align}

Suppose $P$ has $p_1$ prime factors with boundaries.
Since $P$ is orientable, by Proposition \ref{b}, we have 
\begin{align}\label{4.12}
2b(\pi_1(P))= -2p_1+\chi(\partial P).
\end{align}

Suppose  the prime decomposition of $N$ is
$$N=(\#_{i=1}^v V_i)\#(\#_{j=1}^{p_2} Q_j)\#(\#_{e=1}^q RP^2\times S^1),$$
where each $V_i$ is closed,   aspherical, and each $Q_j$ is  aspherical with boundary.
Let $\tilde V_i$ and $\tilde Q_j$ be the orientable double covers of $V_i$ and $Q_j$ respectively. Then
$$\tilde N=(\#_{i=1}^v \tilde V_i)\#(\#_{j=1}^{p_2} \tilde Q_j)\#(\#_{e=1}^{v+p_2+2q-1} S^2\times S^1)$$
Since $\tilde N$ is orientable and $\chi(\partial \tilde N)/2=\chi(\partial  N)$, by Proposition \ref{b} we have 
\begin{align}\label{4.13}
b(\pi_1(\tilde N))= -p_2-(v+p_2+2q-1)+\chi(\partial \tilde N)/2
= -(v+2p_2+2q-1)+\chi(\partial  N).
\end{align}

By (\ref{4.3})-(\ref{4.5}), we have
\begin{align}\label{4.14}
b(\pi_1(N))=\Sigma_{i=1}^v b(V_i)+\Sigma_{j=1}^{p_2} b(Q_j)+\Sigma_{e=1}^q b(RP^2\times S^1)= -(v+p_2+q)+\chi(\partial N)/2.
\end{align}

By (\ref{4.13}) and (\ref{4.14})  we have 
\begin{align}\label{4.15} 
\begin{split}
&|-2 b(\pi_1(N))+b(\pi_1(\tilde N))-1|\\
=\ &|2(v+p_2+q)-\chi(\partial N)-(v+2p_2+2q-1)+\chi(\partial N)-1|=v
\end{split}
\end{align}

Substitute (\ref{4.12}), (\ref{4.13}) and (\ref{4.15}) into (\ref{4.10}), we have 
\begin{align}\label{4.16}
\begin{split}
&\chi_4(G)\geq 1+(-2p_1+\chi(\partial P)) +(-(v+2p_2+2q-1)+\chi(\partial  N))+v\\
=\ & 2-2(p_1+p_2)- 2q+ \chi(\partial P)+\chi(\partial  N)=2-2(p+q)+\chi(\partial  M).
\end{split}
\end{align}
Then by Proposition \ref{main1}, we have 
$$\chi_4(G)= 2-(p+q)+\chi(\partial M).$$
\end{proof}

{\bf Acknowledgments.} The referee provides insightful and efficient suggestions 
which greatly simplify the proofs and computations.
We thank the referee for the advice.

We also thank Professor Jonathan Hillman, Professor Dieter Kotschick, Professor Yang Su and Professor Shicheng Wang for their comments on literature and known results in this topic.


\begin{thebibliography}{BHJSM}





\bibitem[AH]{AH}  A. Adem, I. Hambleton,
{\it Minimal Euler characteristics for even-dimensional manifolds with finite fundamental group,} preprint 2021, \url{https://arxiv.org/abs/2103.00703}.

\bibitem[BK]{BK} S. Baldridge and P. Kirk, {\it On symplectic 4-manifolds with prescribed fundamental group,} Commentarii Math. Helv. 82 (2007) 845-875.


\bibitem[Ha1]{Ha1} A. Hatcher, {\it Algebraic topology.} Cambridge University Press, Cambridge, 2002.

\bibitem[Ha2]{Ha2} A. Hatcher, {\it Notes on basic $3$-manifold topology.} \url{http://pi.math.cornell.edu/~hatcher/}.

\bibitem[HW]{HW} J. Hausmann, S. Weinberger, 
{\it Caractéristiques d'Euler et groupes fondamentaux des varietes de dimension 4. (French) [Euler characteristics and fundamental groups of 4-manifolds]} 
Comment. Math. Helv. 60 (1985), no. 1, 139-144.


\bibitem[He]{He}
J. Hempel, {\it $3$-manifolds,} Princeton University Press and University
of Tokyo Press, 1976.

\bibitem[Hi]{Hi}  J. A. Hillman, {\it Four-manifolds, geometries and knots.} Geometry $\&$ Topology Monographs, 5. Geometry $\&$ Topology Publications, Coventry, 2002

\bibitem[JK]{JK} F. Johnson, D. Kotschick,  {\it On the signature and Euler characteristic of certain four-manifolds.} Math. Proc. Cambridge Philos. Soc. 114 (1993), no. 3, 431-437. 

\bibitem[KLPT]{KLPT} D.  Kasprowski, M. Land, M. Powell, P. Teichner, {\it Stable classification of 4-manifolds with 3-manifold fundamental groups.} J. Topol. 10 (2017), no. 3, 827-881.


\bibitem[Kir]{Kir} R. C. Kirby, {\it The Topology of 4-Manifolds}, Lecture Notes in
Math. {\bf 1374}, Springer-Verlag, Berlin (1990).


\bibitem[KL]{KL} P. Kirk, C. Livingston,  {\it The geography problem for 4-manifolds with specified fundamental group.} Trans. Amer. Math. Soc. 361 (2009), no. 8, 4091-4124. 

\bibitem[Ko1]{Ko1} D. Kotschick, 
{\it Four-manifold invariants of finitely presentable groups.} Topology, geometry and field theory, 89–99, World Sci. Publ., River Edge, NJ, 1994. 

\bibitem[Ko2]{Ko2} D. Kotschick, {\it Minimizing Euler characteristics of symplectic four-manifolds.} Proc. Amer. Math. Soc. 134 (10) (2006), 3081-3083

\bibitem[We]{We} C. Weibel, {\it An introduction to homological algebra.} Cambridge Studies in Advanced Mathematics, 38. Cambridge University Press, Cambridge, 1994. xiv+450 pp.

\end{thebibliography}
\end{document}